\documentclass[a4paper,12pt]{article}
\usepackage{amsfonts,amsthm,amsmath}
\newtheorem{thm}{Theorem}
\newtheorem{prop}[thm]{Proposition}

\newtheorem{cor}[thm]{Corollary}
\theoremstyle{remark}
\newtheorem{rem}[thm]{Remark}
\newcommand{\FF}{\mathbb{F}}
\newcommand{\ZZ}{\mathbb{Z}}
\newcommand{\RR}{\mathbb{R}}

\begin{document}

\title{Frames in the odd Leech lattice}

\author{
Tsuyoshi Miezaki\thanks{Department of Mathematics, 
Oita National College of Technology, 
1666 Oaza-Maki, Oita 870-0152, Japan. email: miezaki@oita-ct.ac.jp
}
}

\date{}

\maketitle

\begin{abstract}
In this paper, we show that 
there is a frame of norm $k$ in the 
odd Leech lattice for every $k\ge 3$. 
\end{abstract}

{\small

\quad
}

\section{Introduction}
Let $\Lambda$ be an $n$-dimensional (Euclidean) lattice 
in $\RR^{n}$. 
The dual lattice $\Lambda^*$ of $\Lambda$ is the lattice
$
\{x\in \RR^{n} \mid (x,y) \in\ZZ \text{ for all }
y\in \Lambda\}
$, 
where $(x, y)$ is the standard inner product.
A lattice $\Lambda$ with $\Lambda=\Lambda^{*}$
is called {\em unimodular}.
The norm of a vector $x$ is defined as $(x, x)$.
A unimodular lattice with even norms is said to be {\em even}, and that containing a vector of odd norm is said to be
{\em odd}.
An $n$-dimensional even unimodular lattice exists if and only
if $n \equiv 0 \pmod 8$, while an odd unimodular lattice
exists for every dimension. 
For example, the unique even unimodular lattice without roots 
in $24$ dimensions is the Leech lattice $\Lambda_{24}$ and 
the unique odd unimodular lattice without roots 
in $24$ dimensions is the odd Leech lattice $O_{24}$ \cite{SPLAG}. 
The minimum norm $\min(\Lambda)$ of $\Lambda$ is the smallest
norm among all nonzero vectors of $\Lambda$.



Let $\ZZ_{k}$ be the ring 
of integers modulo $k$, where $k$ 
is a positive integer. 
In this paper, we always assume that $k\geq 3$ and 
we take the set $\ZZ_{k}$ to be 
$\{0,1,\ldots,k-1\}$.
A $\ZZ_{k}$-code $C$ of length $n$
(or a code $C$ of length $n$ over $\ZZ_{k}$)
is a $\ZZ_{k}$-submodule of $\ZZ_{k}^n$.
A $\ZZ_{k}$-code $C$ is {\em self-dual} if $C=C^\perp$, where
the dual code $C^\perp$ of $C$ is defined as 
$C^\perp = \{ x \in \ZZ_{k}^n \mid x \cdot y = 0$ for all $y \in C\}$
under the standard inner product $x \cdot y=\sum_{i=1}^nx_iy_i$. 
The Euclidean weight of a codeword 
$x=(x_1,x_2,\ldots,x_n)$ is 
$\sum_{i=1}^n \min\{x_i^2,(k-x_i)^2\}$.
The minimum Euclidean weight $d_E(C)$ of $C$ is the smallest Euclidean
weight among all nonzero codewords of $C$.

We now give a method to construct 
unimodular lattices from self-dual $\ZZ_{k}$-codes which 
is referred to as Construction A~\cite{{BDHO},{HMV}}. 
Let $\rho$ be a map from $\ZZ_{k}$ to $\ZZ$ sending $0, 1, \ldots , k$ 
to $0, 1, \ldots , k$. 
If $C$ is a  self-dual $\ZZ_{k}$-code of length $n$, then 
the lattice 
\[
A_{k}(C)=\frac{1}{\sqrt{k}}\{\rho (C) +k \ZZ^{n}\} 
\]
is an $n$-dimensional unimodular lattice, where 
\[
\rho (C)=\{(\rho (c_{1}), \ldots , \rho (c_{n}))\ 
\vert\ (c_{1}, \ldots , c_{n}) \in C\}. 
\]
The minimum norm of $A_{k}(C)$ is $\min\{k, d_{E}(C)/k\}$.


A set $\{f_1, \ldots, f_{n}\}$ of $n$ vectors $f_1, \ldots, f_{n}$ in an
$n$-dimensional unimodular lattice $L$ with
$ ( f_i, f_j ) = k \delta_{i,j}$
is called a {\em $k$-frame} of $L$,
where $\delta_{i,j}$ is the Kronecker delta.
It is known that an even unimodular lattice $L$
contains a $k$-frame if and only if there exists 
a Type~II $\ZZ_{k}$-code $C$ 
such that $A_{k}(C) \simeq L$ \cite{{Chapman},{HMV}}. 

Chapman, Gulliver, and Harada showed that 
there exists a $2k$-frame in the 
Leech lattice for every $k\geq 2$ \cite{{Chapman},{GH01}}. 
This gives rise to a natural question, 
is there a $k$-frame in the odd Leech lattice? 
Harada and Kharaghani showed that there is a $k$-frame 
in the odd Leech lattice for $k\leq 97$ \cite{HK}. 
The aim of the present paper is to settle this problem: 
\begin{thm}\label{main1}
There exists an orthogonal frame of norm $k$ in the 
odd Leech lattice for every $k\geq 3$. 
\end{thm}
\begin{rem}
The even sublattice of the odd Leech lattice 
has index two in the Leech lattice. 
Therefore, Theorem \ref{main1} generalizes the results of 
Chapman, Gulliver, and Harada \cite{{Chapman},{GH01}}. 
\end{rem}
All computer calculations in this paper
were done using {\sc Magma}~\cite{Magma}.

\section{Proof of Theorem \ref{main1}}
\subsection{$k$-frame with $k\neq 11$}
We review the construction of the Leech lattice due to McKay. 
For the details we refer to \cite{Chapman}. 
Let
{\small
\[
S=\left(\begin{array}{ccccccccccccc}
 0&1&1&1&1&1&1&1&1&1&1&1\\
-1&0&1&-1&1&1&1&-1&-1&-1&1&-1\\
-1&-1&0&1&-1&1&1&1&-1&-1&-1&1\\
-1&1&-1&0&1&-1&1&1&1&-1&-1&-1\\
-1&-1&1&-1&0&1&-1&1&1&1&-1&-1\\
-1&-1&-1&1&-1&0&1&-1&1&1&1&-1\\
-1&-1&-1&-1&1&-1&0&1&-1&1&1&1\\
-1&1&-1&-1&-1&1&-1&0&1&-1&1&1\\
-1&1&1&-1&-1&-1&1&-1&0&1&-1&1\\
-1&1&1&1&-1&-1&-1&1&-1&0&1&-1\\
-1&-1&1&1&1&-1&-1&-1&1&-1&0&1\\
-1&1&-1&1&1&1&-1&-1&-1&1&-1&0
\end{array}\right)_.
\]
}
Let $C$ be the linear code of length $24$ over $\ZZ_4$ 
with generator matrix $M=(I_{12}, 2I_{12}+S)$, 
where $I_n$ is the identity matrix of size $n$. 
Then $A_4(C)\simeq \Lambda_{24}$ \cite{Chapman}. 
Let $D$ be the linear code of length $24$ over $\ZZ_4$ 
with generator matrix $N=(I_{12}, S)$. 
Then we have the following proposition: 
\begin{prop}
$A_4(D)\simeq O_{24}$. 
\end{prop}
\begin{proof}
Since 
\[
NN^{T}=12I_{12}, 
\]
and $|D|=4^{12}$, 
$D$ is a self-dual code, namely, 
$A_4(D)$ is a unimodular lattice. 
Then using {\sc Magma}, we have verified that $A_4(D)$ has no roots, 
which implies that $A_4(D)$ is the odd Leech lattice. 
\end{proof}
Since $S^T=-S$ and $D$ is self-dual, 
$(S, I_{12})$ is also a generator matrix 
for $D$. Then for $a\equiv d\pmod{4}$ and $b\equiv c\pmod{4}$, 
the rows of the following matrix, 
\[P=
\begin{pmatrix}
aI+bS&cI+dS\\
-cI+dS&aI-bS
\end{pmatrix},
\]
are in $O_{24}$. Moreover, 
since $PP^T=(a^2+11b^2+c^2+11d^2)I_{24}$, 
the odd Leech lattice contains a $k$-frame of 
$24$ vectors, where $k=(a^2+11b^2+c^2+11d^2)/4$. 

\begin{thm}\label{thm:knot11}
For each odd prime $p\neq 11$, there exists $a,\ b,\ c,\ d\in \ZZ$ with 
$a\equiv d\pmod{4}$ and $b\equiv c\pmod{4}$, where 
$p=(a^2+11b^2+c^2+11d^2)/4$. 
\end{thm}
\begin{proof}
Let 
$L=\{(a, b, c, d)\in\ZZ\mid a\equiv d\pmod{4} \mbox{ and } b\equiv c\pmod{4}\}$
be the $4$-dimensional lattice with the inner product $\langle ,\rangle$ 
induced by $(a^2+11b^2+c^2+11d^2)/4$. This lattice is spanned by 
$(4,0,0,0)$, $(0,4,0,0)$, $(1,0,0,1)$, and $(0,1,1,0)$, and 
the Gram matrix is as follows: 
\[M:=
\begin{pmatrix}
 4& 0& 1& 0\\
 0& 44& 0& 11\\
 1& 0& 3& 0\\
 0& 11& 0& 3
\end{pmatrix}_.
\]
Then the theta series of $L$, 
\begin{align}
&\theta_L(z)=\sum_{x\in L}q^{\langle x,x\rangle}=\sum_{n=0}^{\infty}a(n)q^n \nonumber\\
&=1 + 4 q^3 + 4 q^4 + 4 q^5 + 4 q^6 + 8 q^7 + 12 q^8 + 12 q^9 +4 q^{10} 
+ 16 q^{12} + 8 q^{13} + \cdots \label{eqn:theta},
\end{align}
is a modular form for $\Gamma_0(44)$, 
because $11M^{-1}$ has integer entries and 
$\det M=11^2$ {\cite[page 192]{Miyake}}. 
Let 
\begin{align*}
\eta(z)=q^{\frac{1}{24}}\prod_{n=1}^{\infty}(1-q^n)
\end{align*}
be the Dedekind $\eta$-function. 
Then $\eta(4z)^8/\eta(2z)^4$ (resp. $\eta(z)^2\eta(11z)^2$) 
is a modular form for $\Gamma_0(4)$ (resp. $\Gamma_0(11)$) 
{\cite[page 145 \& 130]{Kob}}. 
Moreover, by \cite[page 145]{Kob}, we have 
\[
\frac{\eta(4z)^8}{\eta(2z)^4}=\sum_{n=1}^{\infty}\sigma_1(2n-1)q^{2n-1}, 
\]
where $\sigma_1(n)=\sum_{p|n}p$. 
Let $\eta(z)^2\eta(11z)^2=\sum_{n=1}^{\infty}b(n)q^n$ and 
\[
\chi_p(n)=
\left\{
\begin{array}{cl}
0 & \mbox{ if }n\equiv 0\pmod{p}\\
1 & \mbox{ otherwise. }
\end{array}
\right.
\]
Then 
\[
\theta_L(z)_{\chi_2}=\sum_{n=0}^{\infty}\chi_{2}(n)a(n)q^n
=4 q^3 + 4 q^5 + 8 q^7 + 12 q^9 
+ 8 q^{13} + 20 q^{15} + 16 q^{17} + \cdots
\] 
is a modular form for 
$\Gamma_0(176)$~{\cite[page~127]{Kob}}. 
On the other hand, the function 
\[
\frac{4}{5}\left(\frac{\eta(4z)^8}{\eta(2z)^4}-\eta(z)^2\eta(11z)^2\right)
=\sum_{n=1}^{\infty}(\sigma_1(n)-b(n))q^n 
\]
is a modular form for $\Gamma_0(44)$. Then, 
\begin{align*}
\left(\frac{4}{5}
\left(\frac{\eta(4z)^8}{\eta(2z)^4}-\eta(z)^2\eta(11z)^2
\right)
\right)_{\chi_{11}}
&=\sum_{n=1}^{\infty}\chi_{11}(n)(\sigma_1(n)-b(n))q^n\\ 
&=4 q^3 + 4 q^5 + 8 q^7 + 12 q^9 
+ 8 q^{13} + \cdots
\end{align*}
is a modular form for $\Gamma_0(5324)$~{\cite[page~127]{Kob}}. 
Using Theorem 7 in \cite{Murty} and the fact that 
the genus of $\Gamma_0(5324)$ is $694$, and 
checking that for $n\leq 1388$, 
\[
\chi_{2}(n)a(n)=\chi_{11}(n)(\sigma_1(n)-b(n)), 
\]
we have 
\[
\theta_L(z)_{\chi_2}={\left(\frac{4}{5}
\left(\frac{\eta(4z)^8}{\eta(2z)^4}-\eta(z)^2\eta(11z)^2
\right)
\right)_{\chi_{11}}}_, 
\]
namely, for prime $p$ with $p\geq 3$, 
\[
a(p)=\frac{4}{5}(\sigma_1(p)-b(p))=\frac{4}{5}(p+1-b(p)). 
\]
It is known that 
\begin{equation}
|b(p)|<2\sqrt{p} \label{eqn:Deligne}
\end{equation}
for all primes $p$~{\cite[page 297]{Chapman}}. 
Thus, we have 
\[
a(p)=\frac{4}{5}(p+1-b(p))>\frac{4}{5}(p+1-2\sqrt{p})
=\frac{4}{5}(\sqrt{p}-1)^2>0.
\]
\end{proof}
By Theorem \ref{thm:knot11} and Lemma 5.1 in \cite{Chapman}, 
the odd Leech lattice contains a $k$-frame 
if $k\geq 3$, $k\neq 4$, and $k\neq 11m$ for all integers $m\geq 1$. 
However, because the Fourier coefficient $a(4)$ is positive, 
there exists a $4$-frame in the odd Leech lattice. 
Therefore, 
the odd Leech lattice contains a $k$-frame 
if $k\geq 3$, and $k\neq 11m$ for all integers $m\geq 1$.

\subsection{$11$-frame}
Using the quasi-twisted construction \cite{GH01}, 
we have found a $\ZZ_{11}$-code $C_{11}$ of length $24$ 
which has generator matrix $(I_{12}, A)$, where 
{\small
\[
A=\left(\begin{array}{ccccccccccccc}
 2&2&2&10&4&9&7&1&1&1&1&1\\
 10&2&2&2&10&4&9&7&1&1&1&1\\
 10&10&2&2&2&10&4&9&7&1&1&1\\
 10&10&10&2&2&2&10&4&9&7&1&1\\
 10&10&10&10&2&2&2&10&4&9&7&1\\
 10&10&10&10&10&2&2&2&10&4&9&7\\
 4&10&10&10&10&10&2&2&2&10&4&9\\
 2&4&10&10&10&10&10&2&2&2&10&4\\
 7&2&4&10&10&10&10&10&2&2&2&10\\
 1&7&2&4&10&10&10&10&10&2&2&2\\
 9&1&7&2&4&10&10&10&10&10&2&2\\
 9&9&1&7&2&4&10&10&10&10&10&2
\end{array}\right)_.
\]
}
Since $AA^{T}=10I_{12}$, $C_{11}$ is self-dual and 
$A_{11}(C_{11})$ is unimodular. 
Moreover, we have verified using {\sc Magma} that 
$A_{11}(C_{11})$ has minimum norm $3$, which implies 
that $A_{11}(C_{11})$ is the odd Leech lattice. 
The code $C_{11}$ gives an orthogonal frame of norm $11$ 
in the odd Leech lattice by Construction A because 
$A_{11}(C_{11})$ contains the vectors 
$(\sqrt{11},0,\ldots,0)$, $(0,\sqrt{11},0,\ldots,0)$,\ldots,
$(0,\ldots,0,\sqrt{11})$. 
By Lemma 5.1 in \cite{Chapman}, the above frame gives an orthogonal frame 
of norm $11m$ for every $m$. 

Summarizing this section, we proved Theorem \ref{main1}. 
\begin{rem}
A similar argument can be found in \cite{{Chapman},{GH01}} 
for the Leech lattice. 
\end{rem}
Since the existences of an orthogonal frame of norm $k$ in the odd Leech 
lattice and a self-dual $\ZZ_{k}$-code of length $24$ 
are equivalent \cite{{Chapman},{HMV}}, we have the following: 
\begin{cor}
The odd Leech lattice can be constructed using some self-dual 
$\ZZ_{k}$-code by Construction A for any $k\geq 3$. 
\end{cor}

\begin{rem}
Harada and Kharaghani constructed, via Construction A, the odd Leech lattice 
using self-dual codes over $\FF_p$ for $p\leq 97$ \cite{HK}. 
\end{rem}

\bigskip
\noindent
{\bf Acknowledgment.}
The author would like to thank Masaaki Harada for
useful discussions.


\end{document}